\documentclass[11pt,reqno]{amsart}
\usepackage{color} 
\usepackage{amssymb,amsmath,mathrsfs,amsthm}
\usepackage{caption}
\usepackage{enumitem, xcolor}
\usepackage{mathtools}
\usepackage{geometry}
\usepackage{amsfonts}

\usepackage{comment}

\allowdisplaybreaks
\newtheorem{thm}{Theorem}

\newtheorem{lem}[thm]{Lemma}

\theoremstyle{definition}

\newtheorem{rem}{Remark}

\def \no#1#2#3 {{\bf #1} (#3), #2.}
\def \eds#1#2#3 {#1, #2, #3.}

\newcommand{\im}{\mathrm{Im}}

\newcommand{\tp}{\widetilde{\phi}}

\newcommand{\ta}{\widetilde{\alpha}}

\def\R{{\mathbb R}}

\def\d{{\rm d}}

\def\T{{\mathbb T}}

\def\:{{\colon}}

\def\be#1{\begin{equation}\label{#1}}
\def\ee{\end{equation}}

\def\<{\langle}
\def\>{\rangle}
\def\coloneqq{:=}

\newcommand{\lec}{\lesssim}
\newcommand{\bs}{\begin{split}}
\newcommand{\essss}{\end{split}}

\newcommand{\eqnb}{\begin{equation}}
\newcommand{\eqne}{\end{equation}}

\renewcommand{\ee}{\mathrm{e}}

\newcommand{\p}{\partial}

\newcommand{\ii}{{\rm in}}
\newcommand\out{{\rm out}}

\renewcommand{\T}{\mathbb{T}}
\renewcommand{\R}{\mathbb{R}}

\newcommand{\C}{\mathbb{C}}

\newcommand{\Z}{\mathbb{Z}}

\renewcommand{\d}{\mathrm{d}}

\begin{document}

\title[Instability of shear flows]{A simple proof of linear instability of shear flows with application to vortex sheets}
\author[A.~Kumar]{Anuj Kumar}
\address{Department of Mathematics, Florida State University, Tallahassee, FL 32306}
\email{ak22cm@fsu.edu}
\author[W.~S.~O\.za\'nski]{Wojciech O\.za\'nski}
\address{Department of Mathematics, Florida State University, Tallahassee, FL 32306}
\email{wozanski@fsu.edu}
\maketitle

\date{}

\medskip

\begin{abstract}
We consider the construction of  linear instability of parallel shear flows, which was developed by Zhiwu Lin (\emph{SIAM J. Math. Anal.} 35(2), 2003). We give an alternative simple proof in Sobolev setting of the problem, which exposes the mathematical role of  the Plemelj-Sochocki formula in the emergence of the instability, as well as does not require the cone condition. Moreover, we localize this approach to obtain an approximation of the Kelvin-Helmholtz instability of a flat vortex sheet. 
\end{abstract}


\noindent\thanks{\em Keywords:\/}
shear flows, Rayleigh equation, instability 

\section{Introduction}\label{sec_intro}

We consider the 2D incompressible Euler equations
\[
\begin{split} u_t + (u\cdot \nabla ) u + \nabla p &=0\\
\mathrm{div}\,u &=0
\end{split}
\]
for $x\in \T$, $y\in (-1,1)$ with impermeability boundary conditions $u_2(x,-1,t)= u_2(x,1,t)=0$. We linearize the equations around a shear flow $u=(U(y),0)$, and we consider perturbations of the form $v=(\Psi_y , - \Psi_x)$ for the stream function $\Psi $ of the form
\eqnb\label{stream_fcn}
\Psi (x,y,t) = \phi (y) \ee^{i\alpha (x-ct)}.
\eqne
Then the linearized Euler equations become 
\eqnb\label{eq_phi_basic}
-\phi'' + \alpha^2 \phi + \frac{U''}{U-c } \phi =0,
\eqne
often referred to as the \emph{Rayleigh equation},  with the boundary conditions
\eqnb\label{BCs_phi}
\phi (-1)=\phi(1)=0.
\eqne
 We assume that $U\in C^5 ([-1,1])$ is such that 
\eqnb\label{Kgeq0}
- \frac{U''}{U} \geq 0,
\eqne
 and we make the following assumption.\\

\textbf{Assumption A} $U$ has exactly one zero, say, $U(a)=0$ for some $a\in (-1,1)$.\\

We note that then \eqref{Kgeq0} implies that $a$ is the only inflection point of $U$.  We also note that the case of $U$ having finitely many zeros is analogous, and we only consider Assumption A for simplicity. We also note that one can, equivalently, consider an inflection point which is not a zero of $U$, simply by considering $U-U(a)$ instead of $U$.

In the case $c=0$ the problem \eqref{eq_phi_basic}--\eqref{BCs_phi} is of the Sturm-Liouville type, and so there exists the maximal eigenvalue, which we assume to be positive. We will denote it by $\ta^2$, with the corresponding (single) eigenfunction $\tp \in L^2 (-1,1)$, which is positive in $(-1,1)$ (see \cite[Ch.~XI, Theorem~4.1]{hartman}). In particular 
\eqnb\label{phi_nonzero}
\tp (a) \ne 0,
\eqne
and \eqref{eq_phi_basic}--\eqref{BCs_phi} imply that $\tp \in H^1_0 (-1,1)$.
We further note that our assumption of $\ta^2 >0$ can be guaranteed by choosing $U$ appropriately. Namely, by the Minimization Principle, 
\[
-\ta^2 = \min \left\lbrace \int_{-1}^1 |\phi' |^2 + \int_{-1}^1 \frac{U''}U |\phi|^2 \, : \, \phi \in H^1_0, \int_{-1}^1 |\phi |^2 =1 \right\rbrace,
\]
see \cite[(5.6.5)]{haberman} or \cite[(7)]{lin_siam}, we see that $\ta^2 >0$ if $U''/U$ is a sufficiently negative function. This can be guaranteed, for example, by ensuring that $U'''$ is much larger than $U$ in a neighbourhood of $a$.

The problem of linear instability is concerned with finding a solution $\phi$ to \eqref{eq_phi_basic} with some $c\in \C$ with $\im\,c>0$, so that the stream function \eqref{stream_fcn} of the perturbation grows in $t$. In this context the eigenfunction $\tp$ provides a \emph{neutral limiting mode}, i.e. it is a solution with $c=0$, but perturbing this solution (i.e. treating it as the limiting mode) may give us the desired instability. 

To be more precise we now consider a perturbation $(\tp +\psi , \ta^2-\epsilon , -c)$ of $(\tp , \ta^2 , 0)$, in which case the Rayleigh equation becomes 
\eqnb\label{rayleigh}
-\phi'' + \alpha^2 \phi - \varepsilon \phi + \frac{U''}{U-c } \phi =0,
\eqne
and the problem of neutral limiting linear instability becomes: For sufficiently small $\varepsilon $ find $c(\varepsilon )\in \C$ such that, $\im\, c>0$, and the Rayleigh equation \eqref{eq_phi_basic} has a nontrivial solution $\phi$, which we will refer to as the \emph{unstable solution}.

We note that this problem has a long history, with the strategy of determining neutral limiting modes going back to the work of  Tollmien~\cite{tollmien}. We note that a necessary condition for instability is that $U$ has at least one inflection point, which was shown by Lord Rayleigh~\cite{rayleigh}, and later improved by Fj\o rtoft~\cite{fjortoft}. Moreover Howard~\cite{howard_64} estimated the maximal number of possible unstable modes and later Faddeev \cite{faddeev} studied  spectral properties of the Rayleigh equation. Furthermore, Friedlander and Howard~\cite{friedlander_howard} studied the particular velocity profile $U(y) = \cos (my)$ using the method of continued fractions. We also note some modern developments, including Vishik's papers \cite{vishik_1,vishik_2} (and the review article \cite{ABCD}) on instabilities of $2$D radial vortices and the proof of nonuniqueness of the forced 2D Euler equations. A new method of proving linear instability and another proof of Vishik's result was recently developed by Castro~et.~al.~\cite{cfms}, see also a very recent result of Dolce and Mescolini~\cite{dolce_mescolini}, who use a trick of Golovkin~\cite{golovkin} to obtain a simplified proof.   For nonlinear instability we refer the reader to \cite{lin3,friedlanderstraussvishikearly, friedlanderstraussvishik,grenier,dolce_mescolini}\\

We also emphasize that the first rigorous construction of unstable solutions $\phi$ to \eqref{rayleigh} is due to Zhiwu Lin~\cite{lin_siam}. We also note that Lin~\cite{lin_siam} has also extended the curve $\varepsilon \to c(\varepsilon)$ from the neighbourhood of $\varepsilon=0$ to the entire range $(0, \ta^2)$, as well as establishes nonlinear instability in \cite{lin3}. \\

The approach of Lin is based on the shooting method of solving ODEs. To be precise, let us denote by $\phi_1 (y;c,\varepsilon )$ the solution to \eqref{rayleigh} such that $\phi_1 (-1 )=0$ and  $\phi_1'(-1)=\tp'(0)$, and let
\[
I(c,\varepsilon ) \coloneqq \phi_1 (1;c,\varepsilon ).
\]
We note that an expression for $I$ could be derived in terms of the Green's function, using the Wronskian \cite[p.~335]{lin_siam}, and so the problem of finding  $c,\varepsilon$ for which the unstable solution $\phi$ exists reduces to studying zeros of $I$. One can then explicitly compute the expressions for 
\eqnb\label{ders_I}
\p_\varepsilon I\quad \text{ and }\quad \p_c I,
\eqne
see \cite[(37)--(38)]{lin_siam}. Consequently, one can restrict oneself to a sufficiently small triangle contained in the upper half-plane of $\C$ to find the curve of solutions $c=c(\varepsilon)$ for $\varepsilon >0$ close to $0$, using the Banach Contraction Mapping Theorem of the triangle, see \cite[Theorem~4.1]{lin_siam} for details. 

Crucially, the fixed point argument relies on the Plemelj-Sochocki \cite{plemelj,sochocki} formula,
\eqnb\label{plemelj_formula}
\frac{1}{x+ c} \to - i \pi \delta_{x=0} +\mathrm{p.v.} \frac{1}x\qquad \text{ as } c\to 0, \im \,c>0
\eqne
(see Lemma~\ref{L_plemelj} for some details), which is used in the computation of \eqref{ders_I} in the limit $c,\varepsilon \to 0$, see~\cite[(41)]{lin_siam}. One can think of the formula \eqref{plemelj_formula} as the fundamental reason why the instability occurs.  This can be seen through the asymptotic expansion
\eqnb\label{c_expansion}
c(\varepsilon ) = -\frac{\varepsilon }{\lambda } + o (\varepsilon )\qquad \text{ as } \varepsilon \to 0^+,
\eqne
which follows from \eqref{reduced_new}, \eqref{exp_G} below, where 
\[
\lambda \coloneqq \lim_{\im \,c>0, \,c\to 0} \int_{-1}^1 \frac{-U''(y) |\tp (y)|^2 }{(U(y)-c)U(y)} \d y.
\]
Indeed, we will use the Plemelj-Sochocki formula \cite{plemelj,sochocki} (in \eqref{limit_of_gamma} below) to see that,
\eqnb\label{lambda_fact}
\lambda = C - i\pi \frac{U''(a)\tp (a)^2}{|U'(a)|}
\eqne
for some $C\in \R$. Thus $\im\frac{-1}\lambda  >0$, and hence \eqref{c_expansion} shows that the solution curve $c(\varepsilon )$ moves inside the upper half-plane of $\C$ as soon as $\varepsilon $ becomes positive. Note that this part vanishes if the limit is taken along $c\in \R$. 

We note that the formula \eqref{lambda_fact} is classical, and in fact $C= -\mathrm{p.v.} \int_{-1}^1 \frac{U''(y)\tp (y)^2}{U(y)^2}\d y$, see for example \cite[p.~13, (2.20)]{drazin_howard} or Lin \cite[(35)]{lin_siam}, who proved it by dividing the two quantities in \eqref{ders_I}\footnote{Our proof \eqref{limit_of_gamma} of \eqref{lambda_fact} uses merely the Dominated Convergence Theorem (and \eqref{plemelj_formula}), which makes it a bit  simpler.}.

We emphasize that it is not clear to what extent the neutral limiting instabilities can be localized in space, since the eigenvalue problem \eqref{rayleigh}, at $\varepsilon=0$, $c=0$, is global. On the other hand, the Rayleigh criterion (of existence of an inflection point) is entirely local. Some insight into this problem was recently provided by Liu and Zeng \cite{liu_zeng}, who apply a shooting method and a gluing procedure across a few ``layers'' of a Rayleigh-type equation arising in the context of capillary waves. \\

In this note we first offer an alternative proof of existence of unstable solutions to \eqref{rayleigh}.

\begin{thm}[Linear instability of shear flow in Sobolev spaces]\label{T01}
For all sufficiently small $\varepsilon >0$ there exists $c=c(\varepsilon ) \in \C$ with $\im\,c>0$ such that the Rayleigh equation \eqref{rayleigh} has a solution $\phi\in H^1_0 (-1,1)$ such that $\phi = \tp +\psi$ with $\| \psi \|_{H^1} \lec |c| + \varepsilon$.
\end{thm}

The proof is inspired by a recent work of Vishik \cite{vishik_1,vishik_2}  on nonuniqueness of the forced $2$D Euler equations, and its review paper \cite{ABCD}. This includes analysis of the Rayleigh equation arising from $2$D vortices. To be precise, Chapter~4 of the review paper \cite{ABCD} involves analysis of the Rayleigh equation \eqref{rayleigh}, posed on $\R$, in $L^2$ and H\"older spaces $C^\gamma$, $\gamma \in (0,1)$.  In Theorem~\ref{T01} we consider instead the Sobolev space $H^1_0$, which simplifies the analysis.

In order to expose the main idea, we now briefly describe the strategy of the proof of Theorem~\ref{T01}, which is inspired by \cite[Section~4.7]{ABCD}. 

We first observe that we expect a solution $\phi$ to \eqref{rayleigh} remain close to $\tp$ as $c,\varepsilon\to 0$, where we assume that $\| \tp \|=1$. We thus write
\[
\phi = \tp + \psi,
\]
where we expect that $\int \psi \tp =0$. We define the projection $P \colon L^2 \to L^2$ onto $\mathrm{Span}\{ \tp \}$, 
\[
P \phi \coloneqq (\phi ,\tp) \tp.
\]
We note that the Rayleigh equation \eqref{rayleigh} can then be decomposed into the   \emph{projected equation}
\eqnb\label{proj}
-\psi'' + \ta^2 \psi + \frac{U''}{U} \psi + P \psi  = \varepsilon \phi  + \left( \frac{U''}U - \frac{U''}{U-c} \right) \phi 
\eqne
coupled with the reduced equation
\eqnb\label{reduced}
P\psi =0.
\eqne
This way, we reduce \eqref{rayleigh} to first finding solution $\phi_{\varepsilon ,c}$ to the infinitely-dimensional system \eqref{proj} for every sufficiently small $\varepsilon ,c$, and then we use the reduced, one-dimensional problem \eqref{reduced} to determine the relation $c=c(\varepsilon )$.  In particular  \eqref{reduced} can be written as 
\[
G(\varepsilon , c ) =0, \qquad \text{ where }\,\,\, G(\varepsilon , c )  \coloneqq \int \psi \tp .
\]
Similarly to \cite[Lemma~4.8.1]{ABCD}, we expand $G$ into a local expansion at $0$ to obtain
\[
G(\varepsilon , c ) = c\lambda + \varepsilon + o(\varepsilon + |c|)
\]
as $\varepsilon ,c\to 0$, $\im\, c>0$.
In order to obtain instability we must have $\im \,\lambda >0$ (see \eqref{tildec}),  which can be obtained by the Plemelj-Sochocki formula \eqref{plemelj_formula} as the leading order term of $(U-c)^{-1}$ appearing in \eqref{rayleigh} (see Lemma~\ref{L02}). 
Assuming that $\im\, \lambda >0$ the reduced equation \eqref{reduced} can be solved using a simple application of Rouch\'e's Theorem, see the argument below Lemma~\ref{L06}. \\

Although this argument resembles the proof in \cite[Chapter~4]{ABCD} we note that it is  simpler and shorter. For instance, we do not need any estimates in H\"older spaces. Moreover, our proof of solvability of the projected equation \eqref{proj} follows from an argument based on summability of the Neumann series in $H^1_0$ (see \eqref{neumann} below), avoiding any use of the Green's function or the shooting method for solving $2$nd order ODEs. This is possible due to a direct estimate on the decay of the Fourier Sine coefficients of $(U-c)^{-1}$ (see Lemma~\ref{L03}), and the  Plancherel Theorem. As a byproduct, we do not require the cone condition, i.e. the restriction of $c$ to a cone in the upper half-plane of $\C$ with the tip on the real axis, which is crucial in the estimates of \cite{ABCD}, as well as in the original proof of Lin~\cite{lin_siam} in studying \eqref{ders_I} in preparation for the fixed point argument. Moreover, our proof exposes the role of the Plemelj-Sochocki formula \eqref{plemelj_formula}\footnote{This is achieved by Lemma~\ref{L02}, which shows how the key expression $(U-c)^{-1}$ in the Rayleigh equation \eqref{rayleigh} can be approximated, uniformly in $y\in (-1,1)$ by an expression reminiscent of the left-hand side of \eqref{plemelj_formula}, from where an application of \eqref{plemelj_formula} (in~\eqref{limit_of_gamma}) makes clear that $\im\,\lambda <0$, which is responsible for the emergence of the instability, as discussed in \eqref{c_expansion} above.}, which  is central to the emergence of the instability.  \\

Furthermore, the Sobolev setting of our approach can be extended to a version that is localized in space and captures an instability occurring in a neighbourhood of a ``strong'' and localized inflection point of a velocity profile $U(y)$.  This is the subject of our second theorem.

 In order to state it, we suppose that $U(y) = U_0 (ky)$ for some $U_0$ satisfying assumptions listed around \eqref{Kgeq0}  and which gives rise to a neutral mode instability namely there exists a positive eigenvalue $\ta^2$ of \eqref{eq_phi_basic} with $c=0$. We also assume that $U_0(y)=1$ for $y<-2$ and $U_0 (y)=-1$ for $y>2$. We note that, since $U''/U= k^2 U_0''/U_0$ one can observe that $\ta/ k \to \alpha_0$ as $k\to \infty$, where $\alpha_0^2$ is the maximal eigenvalue of \eqref{eq_phi_basic} (with $c=0$) posed on $\R$ (see Lemma~\ref{L_app}), with a single eigenfunction $\Phi_0 \in L^2 (\R)$, i.e.
\eqnb\label{Phi0}
-\Phi_0''+ \alpha_0^2 \Phi_0 
+ \frac{U_0''}{U_0} \Phi_0 =0\quad \text{ in } \R.
\eqne
 For convenience we assume that 
\[
\| \Phi_0 \|_{H^1} =1.
\]

Thus, as $k\to \infty$, in which case $U$ approximates the vortex sheet with constant vorticity at $y=0$, one expects the unstable solution to \eqref{rayleigh} to be approximately equal to $\tp (ky)$, and our second theorem makes this observation precise. 
\begin{figure}[h]
\centering
\includegraphics[width=15cm]{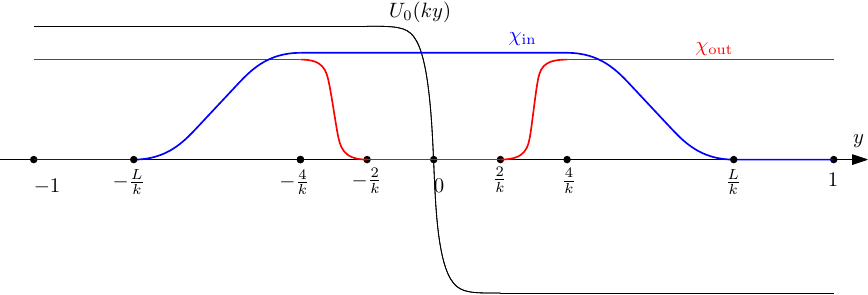}
  \captionof{figure}{A sketch of the velocity profile $U(y)=U_0 (ky) $ and the cutoffs $\chi_{\ii}$, $\chi_{\out}$ leading to the inner and outer estimates \eqref{multisc_ests}.}\label{fig_sketch} 
\end{figure}
In order to introduce it, we let $L>1$ be a large number (to be specified) and we define the inner and outer cutoffs $\chi_{\ii }$, $\chi_{\out}$  to be smooth functions such that $\chi_{\ii} = 1$ on $(-4/k,4/k)$, $\chi_{\ii}=0$ outside $(-L/k,L/k)$, $\chi_{\out} =0$ in $(-2/k,2/k)$ and $\chi_{\out}=1$ outside $(-4/k,4/k)$ and
\eqnb\label{cutoffs}
\chi_{\ii}^{(l)} \leq \left( \frac{10k}L\right)^l, \qquad \chi_{\out}^{(l)} \leq \left( {10k} \right)^l \qquad  \text{ for } l=0,1,2, 
\eqne
Note that $\chi_{\ii}$ and $\chi_{\out}$ overlap, see Figure~\ref{fig_sketch}. 
We also define the \emph{inner function space}
\eqnb\label{def_of_Y}
Y \coloneqq H^1 (\R ),
\eqne
and the \emph{outer function space}
\eqnb\label{def_of_Z}
Z \coloneqq H^1_0 ((-1,1)) ,\qquad \| f \|_Z \coloneqq \sqrt{\frac{L}k}  \| f' \|_{L^2} + \sqrt{kL}  \| f \|_{L^2}.
\eqne

\begin{thm}\label{T02}
There exists $L>1$ and $\varepsilon_0,c_0>0$ with the following property: for all sufficiently large $k$ and all  $\varepsilon \in (0,\varepsilon_0 k^2)$ there exist $c=c(\varepsilon,k)\in \C$ with $\im\, c\sim \varepsilon/k^2>0$ and such that the Rayleigh equation \eqref{rayleigh} has a nontrivial solution $\phi\in H^1_0 (-1,1)$, which admits the  decomposition  
\[
\phi (y) = \phi_{\out }(y) \chi_{\out}(y) + \left( \Phi_0 (ky) + \Psi (ky)  \right) \chi_{\ii }(y),
\]
where $\Psi \in H^1$, $\phi_{\out} \in Z$ satisfy the estimates
\eqnb\label{multisc_ests}
\| \Psi \|_{H^1} + \| \phi_{\out} \|_Z \lec \varepsilon + \left| \frac{\ta^2}{k^2} - \alpha_0^2 \right|  + |c| +L^{-1/2}.
\eqne
\end{thm}
We present the proof of  Theorem~\ref{T02} in Section~\ref{sec_thm2} below. The proof, including the gluing procedure described above and some estimates of the inner and outer equations, is inspired by the recent paper \cite{AO}, which develops  an inner-outer gluing approach for capturing instabilities of the 3D inviscid vortex columns. We also note a related result  \cite{ABC2} on gluing nonunique solutions to the forced Navier-Stokes equations.\\

We note that the claim of Theorem~\ref{T02} admits certain rescaling with respect to $k$. The rescaling is, however, not complete for the domain we consider (i.e. that $y\in (-1,1)$), and so, in order to explain it, let us suppose for the moment that $y\in \R$. In that case it suffices to prove the claim for $k=1$. Indeed, if $(\phi , \alpha , c, \varepsilon )$ is a solution to the Rayleigh equation \eqref{rayleigh} with a shear flow $U$, then $\Upsilon (y ) \coloneqq \phi (k y)$, $a\coloneqq k \alpha $, $c$, $\epsilon \coloneqq k^2 \varepsilon$ solves the Rayleigh equation with shear flow $\mathcal{U}(y) \coloneqq U(k y)$. This shows that the wave speed $c$ remains the same at each scale, while the growth rate $\alpha c$ (recall~\eqref{stream_fcn}) increases as $k\to \infty$. In fact, for a shear flow with characteristic length $1/k$, the growth rate is of the order $k$, which is consistent with the known asymptotics, see for example the remarks at the bottom of Section~3.2 in Grenier \cite{grenier}. This also provides a link between Theorem~\ref{T02} and the Kelvin-Helmholtz instability \cite{drazin_reid,acheson,majda_bertozzi} of vortex sheets. For example, it is well-known in the context of the Kelvin-Helmholtz instability that considering  independent modes of oscillations of a form similar to \eqref{stream_fcn}, the growth rate in time is predicted to be of the size proportional to  $\alpha$, representing mode frequency in $x$ (see \cite[(4.27)]{drazin_reid}, for example). This is consistent with the scaling of Theorem~\ref{T02}, since the growth rate in time is $\alpha\im\,c$ and  the maximal value of $\im\,c$ is proportional to $\varepsilon_0 =O(1)$.  
We also recall the results of Grenier, who proved (in~\cite[Corollary~1.2]{grenier}) nonlinear instability of shear flows of a similar form as those considered in Theorem~\ref{T02}. The nonlinear instability arises from linear instability, and Grenier indicates in \cite[Section~4.3]{grenier} a family of piecewise affine shear profiles which are linearly unstable (together with their regularizations by mollification).  \\

We note that we do not claim any convergence result as $k\to \infty$. In fact, the relation of this limit with the problem of instabilities of vortex sheets remains an interesting open problem.  

\section{Proof of Theorem~\ref{T01}}\label{sec:pf:T1}
 Here we prove Theorem~\ref{T01}.  We first let $K \colon L^2 \to H^1_0\cap H^2$ denote the solution operator of 
\[
-\psi''+\ta^2 \psi = f
\]
with homogeneous Dirichlet boundary conditions at $y=\pm 1$. We note that $K$ is a compact operator from $L^2$ to $H^s$ for any $s\in [0,2)$. Taking $K$ of \eqref{proj} we obtain
\eqnb\label{proj1}
 \psi+K \left(  \frac{U''}{U} \psi+P \psi \right) = K \left( \varepsilon \phi  + \left( \frac{U''}U - \frac{U''}{U-c} \right) \phi  \right) 
\eqne
Denoting the left-hand side by $T\psi$ and the right-hand side by $R_{\varepsilon ,c}$ we obtain
\eqnb\label{proj2}
T \psi= R_{\varepsilon, c }\psi+R_{\varepsilon, c }\tp.
\eqne
In order to solve this equation for all small $\varepsilon,c>0$, we first show some properties of $T$ and $R_{\varepsilon ,c}$.
\begin{lem}\label{L_T}
$T\colon H^1_0 \to H^1_0$ is invertible.
\end{lem}
\begin{proof}
We first note that, since $\phi \mapsto K \left( \frac{U''}{U} \phi + P \phi \right)$ is a compact operator $H^1_0\to H^1_0$, we have that $T$ is Fredholm of index $0$, and so it suffices to show that $\mathrm{ker}\,T=0$. To this end, suppose that $T\phi =0$, that is
\eqnb\label{EQ01}
L\phi + P \phi \coloneqq -\phi'' + \ta^2 \phi + \frac{U''}{U} \phi + P \phi =0
\eqne
for some $\phi \in H^1_0$. Note that $L$ is self-adjoint, and that $L\phi =0$ if and only if $\phi \in \mathrm{Span}\,\{ \tp \}$. Multiplying \eqref{EQ01} by $\tp$ and integrating, we obtain
\[
 (P \phi,\tp )  =- (L \phi, \tp )  =- (\phi, L\tp  )=0, 
\]
and so $P \phi \perp \tp$. Thus $\phi =0$, as required.
\end{proof}

As for $R_{\varepsilon ,c}$ we first consider the term $1/(U-c)$. We note that, since $a\in (-1,1)$ is an inflection point of $U$, we have $U'(a)\ne 0$. Hence, for every sufficiently small $|c|$, there exists $a'\in (-1,1)$ such that $U(a') = c_R$, where $c=c_R + i c_I$. In the next lemma we show that $1/(U-c)$ can be approximated by a term resembling the Plemelj-Sochocki formula \eqref{plemelj_formula}.

\begin{lem}[Approximation lemma]\label{L02} For sufficiently small $|c|$ we have that
\eqnb\label{approx_sing}
\frac{1}{U(y)-c } - \frac{1}{U'(a') (y-a') - ic_I} = O(1)+ i o(1)
\eqne
uniformly for $y\in (-1,1)$, as $|c|\to 0$.
\end{lem}
\begin{proof}
We note that
\[
\frac{1}{U(y)-c } - \frac{1}{U'(a') (y-a') - ic_I} = \frac{U''(\xi ) (y-a')^2/2}{(U(y)-c)(U'(a')(y-a')-ic_I)},
\]
and so
\eqnb\label{conv_to0}
\left| \frac{1}{U(y)-c } - \frac{1}{U'(a') (y-a') - ic_I} \right| \lec \frac{\| U'' \|_{L^\infty (a-\delta , a+\delta )} |y-a'| }{|U(y)-c_R |\inf_{(a-\delta ,a+\delta )} |U'|}\lec 1
\eqne
for $|y-a'|\leq \delta/2$ whenever $a'\in [a-\delta/2, a+\delta /2]$, where $\delta >0$ is sufficiently small so that $U $ is strictly monotone in $(a-\delta,a+\delta )$. On the other hand, the same upper bound (dependent on $\delta$) is trivial in the case $|y-a'|>\delta/2$.

As for the $o(1)$ part in \eqref{approx_sing}, given $\varepsilon >0$ we take $\delta>0$ sufficiently small so that 
\[
\frac12 |U'(a) | \leq | U'(y) | \leq 2 |U'(a)|
\]
for $y\in (a-\delta, a+\delta )$ and we let $\eta >0$ be such that 
\[
\| U'' \|_{L^\infty (a-\eta , a+\eta )} \leq \frac{\varepsilon }{4 |U'(a)|^2}.
\]
Then taking $|c|$ sufficiently small so that $a'\in (a-\eta/2,a+\eta/2 )$ we see from \eqref{conv_to0} that 
\[
\left| \frac{1}{U(y)-c } - \frac{1}{U'(a') (y-a') - ic_I} \right| \lec \frac{4 \| U'' \|_{L^\infty (a-\eta  , a+\eta ) }}{|U'(a)|^2 } \leq \varepsilon
\]
for $|y-a| \leq \eta /2$. For $|y-a | >\eta /2$ we take $|c| $ sufficiently small so that 
\[
\begin{split}
&\left| \im \left( \frac{1}{U(y)-c } - \frac{1}{U'(a') (y-a') - ic_I} \right)\right|  \\
&= \left| \frac{c_I\left( U'(a')(y-a') +U(y) -c_R \right) U''(\xi ) (y-a')^2/2 }{((U(y)-c_R) U'(a')(y-a')-c_I^2)^2+ (c_I(U'(a')(y-a') +(U(y)-c_R )c_I)^2}\right| \\
&\lec_\eta | c_I | \frac{|U'' (\xi )| (y-a')^2}{U'(a')^2 (y-a')^2 + O(|c|)}\lec_\eta |c_I| \leq \varepsilon,
\end{split}\]
as required.
\end{proof}

We can use Lemma~\ref{L02} to obtain logarithmic growth of the Fourier Sine coefficients of $1/(U-c)$.
\begin{lem}\label{L03}
If $|c|$ is sufficiently small then
\eqnb\label{F_coeff}
\left| \int_{-1}^1 \frac{1}{U(y)-c} \sin \frac{k\pi (y+1)}2 \d y \right| \lec \log (1+k )  \eqne
for all $k\geq 0$.
\end{lem}
\begin{proof}
Thanks to Lemma~\ref{L02} it suffices to show that
\eqnb\label{F_coeff_new}
\left| \int_{-1}^1 \frac{1}{U'(a')(y-a') - ic_I } \sin \frac{k\pi (y+1)}2 \d y \right| \lec \log (1+k )
\eqne
We have 
\[
\frac{1}{U'(a') (y-a') -ic_I}  = \frac{U'(a') (y-a')+ic_I}{U'(a')^2(y-a')^2 + c_I^2}.
\]
For the imaginary part, we obtain
\[
\begin{split}
\left| \int_{-1}^1 \frac{c_I}{U'(a')^2 (y-a')^2 +c_I^2 }\sin \frac{k\pi (y+1)}2 \,\d y \right| &\leq  \int_{-1}^1 \frac{|c_I|^{-1}}{(U'(a')(y-a')/c_I )^2+1 }\d y \\
&\leq \frac{1}{|U'(a')|^2} \int_{\R} \frac1{z^2+1} \d z \lec 1. 
\end{split}
\]
As for the real part, we expand $\sin (k\pi (y+1)/2)$ into a Taylor expansion at $a'$ to obtain 
\[\begin{split}
\left| \int_{[-1,1] \cap \{ |y-a'|\leq 2/k\} } \frac{U'(a') (y-a')}{U'(a')^2(y-a')^2 + c_I^2} \sin \,\frac{k\pi (y+1)}2 \d y\right| 
&\lec k \int_{a'-2/k}^{a'+2 /k } \frac{|U'(a')| (y-a')^2 }{U'(a')^2(y-a')^2 + c_I^2}  \d y\\
& \lec k \int_{-2/k}^{2 /k } \frac{y^2 }{y^2 + c_I^2}  \d y \lec 1.
\end{split}\]
For $|y-a'|\geq 2/k$ we have that 
\[
\left| \int_{[-1,1]\cap \{ |y-a'|\geq 2/k\} }   \frac{U'(a') (y-a')}{U'(a')^2(y-a')^2 + c_I^2} \sin \,\frac{k\pi (y+1)}2 \d y  \right| \leq \frac{2}{|U' (a')|} \int_{2 /k }^2 \frac{1}{z} \d z  \lec \log (1+k ),
\]
as required.
\end{proof}

Thanks to the above lemmas we can now establish
\begin{lem}[Remainder norm]\label{L04} For every $c\in \C$, $\varepsilon >0$, $ 
\| R_{\varepsilon , c } \|_{H^1_0 \to H^1_0}  \lec (|c | + |\varepsilon |)$.
\end{lem}
\begin{proof}
Letting $f\coloneqq (U-c)^{-1}$, and $g\coloneqq c U''\phi /U$ we see that 
\[
R_{\varepsilon,c} \phi = K  (fg) - \varepsilon \phi .
\]
The first term on the right-hand side is the solution $\psi$ to $\psi''-\ta^2 \psi = fg$, which we can estimate in $H^1$ by denoting by $h_m$ the $m$-th Fourier Sine coefficient of a function $h\colon (-1,1)\to \C$, and applying the Plancherel theorem, 
\[\begin{split}
\| \psi \|_{H^1}^2 &\sim \sum_{k\in \Z} \frac{1+k^2}{(\ta^2 + k^2 )^2 } \left( \sum_{m\in \Z}  g_m \log (1+|k-m|)    \right)^2 \\
&\lec \sum_{k\in \Z} \frac{1+k^2}{(\ta^2 + k^2 )^2 } \left( \sum_{m\in \Z}  g_m (1+|k|)^{1/4} (1+|m|)^{1/4}    \right)^2 \\
&\lec \sum_{k\in \Z} \frac{(1+k^2)(1+|k|)^{1/2}}{(\ta^2 + k^2 )^2 }  \left( \sum_{m\in \Z}  (1+|m|)^{3/2}  |g_m|^2  (1+|m|)^{1/2}    \right)\left( \sum_{m\in \Z}  \frac{1}{1+|m|^{3/2}} \right)  \\
&\lec\| g \|_{H^1}^2 \sum_{k\in \Z} \frac{(1+k^2)(1+|k|)^{1/2}}{(\ta^2 + k^2 )^2 }  \\
&\lec \| g \|_{H^1}^2 \lec_U |c|^2 \| \phi \|_{H^1}^2,
\end{split}
\] 
which gives the required estimate, where we used Lemma~\ref{L03} in the first line. 
\end{proof}
\begin{rem}
We note that the above proof is the only place where we use the assumption that $U\in C^5$ (which is needed for the last inequality). This regularity is not essential; for example Lin~\cite{lin_siam} used only $C^2$ regularity, although one needs a little bit more (such as $U\in C^{2,\gamma}$ for some $\gamma>0$) in order to make sense of the principal value in Plemelj-Sochocki formula \eqref{plemelj_formula}.  For such regularity an alternative proof of Lemma~\ref{L04} below is needed, requiring $C^\gamma$ estimates, for example, as in \cite[Lemma~4.8.1]{ABCD}.
\end{rem}

The invertibility of $T$ and the estimate on $\| R_{\varepsilon ,c}\|$ let us rewrite the projected equation \eqref{proj} as
\[
\psi = ( T^{-1} \circ R_{\varepsilon, c} ) \psi + (T^{-1} \circ R_{\varepsilon, c}) \tp,
\]
which has a unique solution $\psi$ for every sufficiently small $\varepsilon, |c|$, given by the Neumann series expansion,
\eqnb\label{neumann}
\psi = (I-T^{-1}\circ R_{\varepsilon, c} )^{-1}T^{-1} \circ R_{\varepsilon, c} \tp = \sum_{k\geq 1 } (T^{-1}\circ R_{\varepsilon, c})^k \tp .
\eqne
It thus remains to determine the choice of $\epsilon$, $c$ solving the reduced equation \eqref{reduced}, namely the equation
\eqnb\label{reduced_new}
G(c,\varepsilon )=0,
\eqne
where we denoted by $\psi$ the solution \eqref{neumann} of the projected equation \eqref{proj} for given $\epsilon$, $c$, and we set
\eqnb\label{def_G}
G(c,\varepsilon ) \coloneqq (\psi , \tp )=  (T^{-1}\circ R_{\varepsilon, c} \tp , \tp ) + \sum_{k\geq 2} ((T^{-1}\circ R_{\varepsilon, c})^k \tp , \tp ).
\eqne
We have the following.
\begin{lem}[Expansion of $G$]\label{L06}
    $G$ is holomorphic in $c\in \{ z\in \C \colon \im \, z >0 \}$ and satisfies
    \eqnb\label{exp_G}
    G(c,\varepsilon)=c\lambda +\varepsilon+o(|c|+|\varepsilon|)
    \eqne
    as $|c|,\varepsilon \to 0$, $\im \,c>0$, where $\lambda \in \C$ is such that $\im \, \lambda >0$.
\end{lem}
\begin{proof} We first note that the last sum in \eqref{def_G} can be bounded by $o(|c|+\varepsilon)$,
\eqnb\label{Lower_terms}
\begin{split}
\left| \sum_{k\geq 2} \left( (T^{-1}\circ R_{\varepsilon, c})^k \tp , \tp \right) \right|&\leq  
\sum_{k\geq 2} \| (T^{-1}\circ R_{\varepsilon, c})^k \|_{H^1_0\to H^1_0} \lec \sum_{k\geq 2} \| T^{-1}\|^k \| R_{\varepsilon , c }\|^k \notag\\  & \lec \sum_{k\geq 2} (C(|\varepsilon | + |c |))^k \lec O((|\varepsilon | + |c|)^2).
\end{split}
\eqne
As for the linear term, we have 
\begin{align*}
    \left(T^{-1}\circ R_{\varepsilon, c} \tp , \tp \right)&=\left(T^{-1}\circ K \left(\frac{U''}{U}-\frac{U''}{U-c}\right) \tp , \tp \right)+\varepsilon\left(T^{-1}\circ K\tp ,\tp \right)\\
    &=\left( \left(\frac{U''}{U}-\frac{U''}{U-c}\right) \tp , T^{-1}\circ K \tp \right)+\varepsilon\\
    &=c \Gamma (c) + \varepsilon,
\end{align*}
where we used the fact that $T \tp =K \tp$ (recall~\eqref{proj1}), the fact that  $T^{-1}\circ K$ is self-adjoint\footnote{In order to see it, note that, for every $f\in L^2$ we have that $w=Kf$ iff $-w''+\widetilde{\alpha}^2 w =f$. Also, $\varphi = T^{-1}\circ K f $ iff $T\varphi = Kf$, i.e. $\varphi + K\left( \frac{U''}U \varphi + P\varphi \right) = w$. In other words, $K\left( \frac{U''}U \varphi + P\varphi \right)= w-\varphi$, which means that $w-\varphi$ is a solution to $-(w-\varphi)''+\widetilde{\alpha }^2 (w-\varphi )= \frac{U''}U\varphi +P\varphi$, i.e. $-\varphi'' +\widetilde{\alpha}^2 \varphi + \frac{U''}U \varphi +P\varphi = -w''+\widetilde{\alpha}^2 w =f$. The self-adjointness is now clear, since all operators on the left-hand side are self-adjoint.}, and we set  
\[\Gamma(c)\coloneqq \left(\frac{-U''}{(U-c)U}\tp,\tp\right).\]
It remains to show that 
\[\lim_{c_I>0, c\to 0}\Gamma(c)=\lambda,\]
where $\lambda$ is a complex number with positive imaginary part.

Letting 
\[
h \coloneqq -\frac{U''}{U} | \tp |^2,
\]
we see that $h$ is H\"older continuous on $[-1,1]$, and that
\begin{align*}
\Gamma(c)=\int_{-1}^1 \frac{1}{(U-c)}h\,\d y=&\int_{-1}^1 \frac{1}{(U'(a')(y-a')-ic_I)}h(y)\,\d y\\
&\hspace{2cm}+\int_{-1}^1 \left( \frac{1}{U(y)-c } - \frac{1}{U'(a') (y-a') - ic_I} \right) h(y)\,\d y
\end{align*}
We now note that the last integrand is bounded, with its imaginary part converging pointwise to $0$, due to Lemma~\ref{L02}, and so, applying the Dominated Convergence Theorem and the Plemelj-Sochocki formula \eqref{plemelj_formula} (see also Lemma~\ref{L_plemelj} for more details), we obtain that there exists $C\in \R$ such that
\eqnb\label{limit_of_gamma}
\Gamma (c) \to C + i \pi |U'(a)|^{-1} h(a) \coloneqq \lambda
\eqne
as $c\to 0$, $\im\,c>0$. Note that $\im\,\lambda >0$, as required, since $U'(a) \ne 0$ (recall the comments above Lemma~\ref{L02}), and $h(a)>0$ (recall assumption mentioned around \eqref{Kgeq0} and \eqref{phi_nonzero}).
\end{proof}
With the expansion given by Lemma~\ref{L06} we first note that, if there was no ``$o(|c|+\varepsilon )$'' in \eqref{exp_G}, then the only solution $c=\widetilde{c} (\varepsilon )$ to the reduced equation \eqref{reduced_new} would be
\eqnb\label{tildec}
\widetilde{c} (\varepsilon ) \coloneqq \frac{-\varepsilon }\lambda = \frac{-\varepsilon \lambda_R + i \varepsilon \lambda_I}{|\lambda|^2},
\eqne
which describes a straight line originating from $c=0$ (at $\varepsilon =0$) and propagating into the upper complex plane (so that $\im \,c>0$ as required) as $\varepsilon $ increases. In order to incorporate the ``$o(|c|+\varepsilon )$'' term, we write
\[
G(c,\varepsilon ) = \underbrace{\lambda(a)(c-\widetilde{c})}_{=:f(c)}+\underbrace{o(\varepsilon+|c-\widetilde{c}|)}_{=:g(c)} ,
\]
and we consider $|c-\widetilde{c}|,\varepsilon$ sufficiently small so that 
\[|g(c,\varepsilon )| \leq  \frac{\varepsilon }8 + \frac{|\lambda (c-\widetilde{c})|}4.
\]    
Then, for $|c- \widetilde{c}|= \varepsilon /2|\lambda |$ (with possibly smaller $\varepsilon$, so that the last inequality remains true), we have 
\[|g(c , \varepsilon )| \leq \frac{2|\lambda (c-\widetilde{c}) |}8+ \frac{|\lambda (c-\widetilde{c})|}4 = \frac{|\lambda (c-\widetilde{c})|}2 =   \frac{|f(c)|}2 
\]
Hence
\[
|g(c,\varepsilon ) |< |f(c)| \qquad \text{ for } c\in \p D\left( \widetilde{c}(\varepsilon ), \frac{\varepsilon }{2|\lambda | } \right).
\]
Thus, since, in $D\left( \widetilde{c}(\varepsilon ), {\varepsilon }/{2|\lambda | } \right)$, both $g(\cdot ,\varepsilon )$ and $f$ are holomorphic and $f$ has exactly one zero $\widetilde{c}(\varepsilon )$, Rouch\'e's Theorem implies that $G(\cdot ,\varepsilon )=f+g(\cdot , \varepsilon )$ also has a unique zero in $D\left( \widetilde{c}(\varepsilon ), {\varepsilon }/{2|\lambda | }\right)$, which gives the desired solution of the reduced equation \eqref{reduced_new}. 
\section{Proof of Theorem~\ref{T02}}\label{sec_thm2}
Here we prove Theorem~\ref{T02}. We let $\phi = \phi_{\ii } \chi_{\ii } + \phi_{\out}\chi_{\out}$. Then $\phi $ solves \eqref{rayleigh} if $\phi_{\ii }$, $\phi_{\out}$ solve
\eqnb\label{coupled1a}
\phi_{\ii }'' - \ta^2 \phi_{\ii } + \varepsilon \phi_{in} + \frac{U''}{U -c } \phi_{\ii } = - 2 \phi_{\out}'\chi_{\out}' - \phi_{\out} \chi_{\out}'',
\eqne
\eqnb\label{coupled1b}
\phi_{\out}'' - \ta^2 \phi_{\out} + \varepsilon \phi_{\out}  = - 2 \phi_{\ii }'\chi_{\ii }' - \phi_{\ii } \chi_{\ii }''.
\eqne
Let us apply the rescaling $R\colon L^2 \to L^2$, $\Phi (\xi ) = R \phi (\xi )= \phi ( y)$, where $\xi \coloneqq ky $.
Note that
\eqnb\label{norm_R}
\| R \| = k^{1/2},
\eqne
and that $\ta /k  - \alpha_0\to 0$ as $k\to \infty$ (see Lemma~\ref{L_app}). We can thus rewrite \eqref{coupled1a} as the projection of the rescaled inner equation,
\eqnb\label{eqPsi1}
\begin{split}
-\Psi'' &+ \alpha_0^2 \Psi -   \frac{U''}{U} \Psi + P_0 \Psi = \frac{1}{k^2}  R \left(  2\phi_{\out}' \chi_{\out}' + \phi_{\out} \chi_{\out}'' \right) \\
& \hspace{4cm}+\left(  \frac{\varepsilon}{k^2}  - \left(\frac{\ta^2}{k^2} - \alpha_0^2 \right) + \left( \frac{U''}{U} - \frac{U''}{U-c} \right) \right)(\Psi + \Phi_0 )
\end{split}
\eqne 
together with the reduced equation
\eqnb\label{eq2_reduced}
(\Psi , \Phi_0)=0,
\eqne
where we used the decomposition
\[
\Phi_{\ii}= \Phi_0 + \Psi,
\]
and $P_0\Psi \coloneqq (\Psi , \Phi_0)\Phi_0 $ is the projection onto the neutral eigenfunction $\Phi_0$ of the inner problem (recall~\eqref{Phi0}).

As in Section~\ref{sec:pf:T1} we let $K:L^2 \to H^2$ denote the solution operator of $-\Psi''+\alpha_0^2 \Psi = F$, namely 
\[
KF (\xi ) \coloneqq \frac{1}{2\alpha_0}\int_{\R } \ee^{-\alpha_0|\xi - \eta |} F(\eta ) \d \eta ,
\]
and we set
\[
T \Psi \coloneqq \Psi + K \left(-\frac{U''}U\Psi +P_0 \Psi \right),
\]
and 
\[
R_{\delta ,c} \Psi \coloneqq \left( \delta +\left( \frac{U''}{U} - \frac{U''}{U-c} \right) \right) \Psi
\]
We observe, as in Lemma~\ref{L_T}, that $T\colon H^1 \to H^1$ is invertible and that 
\[\| R_{\delta ,c} \|_{H^1\to H^1} \lec \delta + |c|.
\]
Thus, we can rewrite the projected inner equation \eqref{eqPsi1} as
\eqnb\label{eqPsi2}
\Psi = T^{-1} \circ R_{\delta, c} (\Psi + \Phi_0 ) + \underbrace{\frac{1}{k^2}T^{-1} \circ K \circ R \left( - 2\phi_{\out}' \chi_{\out}' - \phi_{\out} \chi_{\out}'' \right)}_{=: B \phi_{\out}}   ,
\eqne
where $\delta \coloneqq \frac{\varepsilon}{k^2} -\left(\frac{\ta^2}{k^2} - \alpha_0^2 \right)$. 
We note that $\| B \|_{Z \to H^1} \lec L^{-1/2} $, since
\[
\left\| R \left( 2\phi_{\out}' \chi_{\out}' + \phi_{\out} \chi_{\out}'' \right) \right\|  \lec k^{1/2} \left( k \| \phi_{\out }' \|  + k^2 \| \phi_{\out } \| \right) = \frac{k^2}{\sqrt{L}} \| \phi_{\out} \|_Z,
\]
where we used \eqref{norm_R}. Letting 
\[
\mathsf{C} \colon H^1 \to Z, \quad  \mathsf{C}\Upsilon =  g \Leftrightarrow   -g'' + \ta^2 g - \varepsilon g  =  2 (R^{-1} \Upsilon)' \chi_{\ii }' + R^{-1}  \Upsilon \chi_{\ii }'' 
\]
for $\varepsilon < \ta^2/2$, one obtains
\[
k\| (\mathsf{C}\Upsilon )' \| + k^2 \| \mathsf{C}\Upsilon \| \leq \left\| 2 (R^{-1} \Upsilon )'\chi_{\ii }' + R^{-1} \Upsilon \chi_{\ii }'' \right\|\lec k^{1/2} \| \Upsilon' \| \frac{k}{L} + k^{-1/2} \| \Upsilon \| \frac{k^2}{L^2} \lec \frac{k^{3/2}}{L} \| \Upsilon \|_{H^1},
\]
i.e. that $\| \mathsf{C} \|_{H^1 \to Z } \lec L^{-1/2}$.
Thus, setting $\mathsf{A} \coloneqq T^{-1} \circ R_{\delta , c } \colon H^1 \to H^1$, we can rewrite \eqref{coupled1a}--\eqref{coupled1b} as
\eqnb\label{entire_clean}
\begin{pmatrix}
I  & 0  \\ 0 & I
\end{pmatrix}
\begin{pmatrix}
\Psi \\ \phi_{\out}
\end{pmatrix}
+
\begin{pmatrix}
-\mathsf{A} & -\mathsf{B} \\ -\mathsf{C}  &0
\end{pmatrix}
\begin{pmatrix}
\Psi \\ \phi_{\out}
\end{pmatrix}
= \begin{pmatrix}
\mathsf{A} \Phi_0 \\ \mathsf{C} \Phi_0
\end{pmatrix}\qquad \text{ in } H^1\times Z,
\eqne
Since 
\[
\| \mathsf{A} \| \lec  \frac{\varepsilon}{k^2}  + \left| \frac{\ta^2}{k^2} - \alpha_0^2 \right| + |c| ,\qquad \| \mathsf{B} \|,\| \mathsf{C} \| \lec L^{-1/2},
\]
the system \eqref{entire_clean} has a unique solution for all sufficiently small $|\varepsilon|$, $|c|$ and sufficiently large $k$, satisfying estimate \eqref{multisc_ests}. 
In order to solve the reduced equation \eqref{eq2_reduced}, we set, analogously to Lemma~\ref{L06},
\[
G(\delta , c) \coloneqq (\Psi , \Phi_0 ),
\]
where $\Psi $ is the solution to \eqref{entire_clean} with $\delta = \frac{\varepsilon}{k^2} - \left| \frac{\ta^2}{k^2} - \alpha_0^2 \right|$. As below Lemma~\ref{L06} we conclude that for every sufficiently small $\delta$ (i.e. for sufficiently small $\epsilon$ and sufficiently large $k$) there exists $c=c(\delta )$ such that $G(\delta ,c)=0$, which concludes the proof of Theorem~\ref{T02}.

\section*{Conflict of interest statement}

The authors declare that there is no conflict of interest.

\section*{Data availability statement}

The authors declare that there is no data associated with this work.

\section*{Acknowledgements}
The authors are grateful to Michele Dolce, Zhiwu Lin and Chongchun Zeng for interesting discussions.

\appendix

\section{Plemelj-Sochocki formula}
\begin{lem}\label{L_plemelj}
Let $a$ and $b$ be constants such that $a<0<b$. Let $f\in C^{0,\alpha}[a,b]$ for some $\alpha \in (0,1)$. Then
\eqnb \label{Plemelj}
\lim_{(\delta, \epsilon)\to (0,0^{+}) } \int^b_a \frac{f(x)}{x+\delta \pm i\epsilon}\,\d x=\mp i\pi f(0)+ \mathrm{p.v.}\int^b_a \frac{f(x)}{x}\,\d x.
\eqne
\end{lem}
\begin{proof}
We first write 
\begin{align}\label{EQA}
\int^b_a \frac{f(x)}{x+\delta + i\epsilon}\,\d x=\int^b_a \frac{f(x)(x+\delta)}{(x+\delta)^2 + \epsilon^2}\,\d x-\int^b_a \frac{i\epsilon f(x)}{(x+\delta)^2+ \epsilon^2}\,\d x.
\end{align}
Given $\gamma\in (0,1)$ we write the first integral on the right-hand side as 
\begin{align}\label{EQB}
\int^b_a \frac{f(x)(x+\delta)}{(x+\delta)^2 + \epsilon^2}\,\d x= &\int_{(a,b)\setminus (-\gamma,\gamma )}  \frac{f(x)(x+\delta)}{(x+\delta)^2 + \epsilon^2}\,\d x+\notag \\& +\int^\gamma_{-\gamma} \frac{(f(x)-f(-\delta))(x+\delta)}{(x+\delta)^2 + \epsilon^2}\,\d x+f(-\delta)\int^{\gamma}_{-\gamma} \frac{(x+\delta)}{(x+\delta)^2 + \epsilon^2}\,\d x.
\end{align}
For $|x|\geq \gamma $   and $|\delta|<\gamma/2$ we consider the first integral on the right-hand side  by noting that 
\begin{align}\label{EQC}
\left|  \frac{f(x)(x+\delta)}{(x+\delta)^2 + \epsilon^2} \right| , \left| \frac{f(x)}{(x+\delta )^2 + \epsilon^2 } \right| \lec \frac{|f(x)|}{|x+\delta|} +\frac{|f(x)|}{|x+\delta|^2}  \lec  \| f \|_\infty /\gamma^2 ,
\end{align}
and so applying the Dominated Convergence Theorem, we obtain
\[\int_{(a,b)\setminus (-\gamma , \gamma ) } \frac{f(x)(x+\delta)}{(x+\delta)^2 + \epsilon^2}\,\d x\to \int_{(a,b)\setminus (-\gamma , \gamma ) } \frac{f(x)}{x}\,\d x\]
as ${(\delta,\epsilon)\to (0,0^+)}$. For the second term on the right-hand side of \eqref{EQB},  
 we have
\begin{align}\label{EQE}
 \left|\int^\gamma_{-\gamma} \frac{(f(x)-f(-\delta))(x+\delta)}{(x+\delta)^2 + \epsilon^2}\,\d x \right|\lec \int^\gamma_{-\gamma} \frac{|x+\delta|^{1+\alpha}}{(x+\delta)^2 + \epsilon^2}\,\d x\lec \int^\gamma_{-\gamma} {|x+\delta|^{\alpha-1}}\,\d x\lec_\alpha \gamma^\alpha,
 \end{align}
 where we used the H\"older continuity assumption on $f$.
Thus 
\[\limsup_{(\delta,\epsilon)\to (0,0)}\left|\int^\gamma_{-\gamma} \frac{(f(x)-f(-\delta))(x+\delta)}{(x+\delta)^2 + \epsilon^2}\,dx \right|\le C_\alpha\gamma^\alpha.\]
For the last integral in \eqref{EQB}, we have
\[\lim_{(\delta,\epsilon)\to(0,0^+)}\int^{\gamma}_{-\gamma} \frac{f(-\delta)(x+\delta)}{(x+\delta)^2 + \epsilon^2}\,\d x=\lim_{(\delta,\epsilon)\to(0,0^+)}\frac{f(-\delta)}{2}\log\left(\frac{(\delta+\gamma)^2+\epsilon^2}{(\delta-\gamma)^2+\epsilon^2}\right)=0.\]
As for the imaginary part in  \eqref{EQA} we write $ \int^b_a= \int_{-\gamma}^\gamma + \int_{(a,b)\setminus (-\gamma ,\gamma )}$. For the first of the resulting terms we use \eqref{EQC} and the Dominated Convergence Theorem to obtain that
\[
\int_{(a,b)\setminus (-\gamma , \gamma )}  \frac{\epsilon f(x)}{(x+\delta)^2 + \epsilon^2}\,\d x \to 0
\]
as $(\delta , \varepsilon )\to (0,0^+)$. For the second one we see that 
\[ \left|\int^\gamma_{-\gamma} \frac{\epsilon (f(x)-f(-\delta))}{(x+\delta)^2 + \epsilon^2}\,\d x \right|\lec \int^\gamma_{-\gamma} \frac{\epsilon|x+\delta|^{\alpha}}{(x+\delta)^2 + \epsilon^2}\,\d x\lec \gamma^\alpha,\]
and
\[
f(-\delta ) \int_{-\gamma}^\gamma \frac{\epsilon}{(x+\delta)^2+\epsilon} \d x = f(-\delta ) \int_{(-\gamma+\delta)/\epsilon}^{(\gamma+\delta)/\epsilon} \frac{\d z}{1+z^2}  \to \pi f(0)\quad \text{ as } (\delta, c)\to (0,0^+).
\]
Collecting all the terms, we obtain
\[\lim_{(\delta, \epsilon)\to (0,0^{+}) } \int^b_a \frac{f(x)}{x+\delta + i\epsilon}\,\d x= -i\pi f(0)+ \left( \int^{-\gamma}_{a}\frac{f(x)}{x}\,\d x+\int_{\gamma}^{b}\frac{f(x)}{x}\,\d x\right) +O(\gamma^\alpha).\]
Letting $\gamma \to 0$ completes the proof.
\end{proof}
\section{Eigenvalue problems}
Here we consider the eigenvalue problem 
\eqnb\label{evalue_prob_app}
-\Psi'' + \beta^2 \Psi + M(y) \Psi=0 \qquad \text{ on } (-a,a),
\eqne
equipped with homogeneous boundary conditions $\Psi (-a) = \Psi (a) =0$, where $M(y)$ is a nonpositive  function with compact support in $\R$.

\begin{lem}\label{L_app}
Let $\beta_a^2 $ denote the largest eigenvalue of \eqref{evalue_prob_app}. Then $\beta_a^2 \to \beta $ as $a\to \infty$, where $\beta$ is the maximal eigenvalue of \eqref{evalue_prob_app} with $a=\infty$.
\end{lem}
\begin{proof}
See \cite[Chapter~II]{titchmarsh}
\end{proof}

\bibliographystyle{plain}
\bibliography{literature}

\end{document}